\documentclass[12pt]{amsart}

\usepackage{amsmath,amsthm,amssymb,amscd}

\usepackage{maplestd2e}

\theoremstyle{plain}
\newtheorem{thm}{theorem}[section]
\newtheorem{theorem}[thm]{Theorem}
\newtheorem{prop}[thm]{Proposition}

\newtheorem{corollary}[thm]{Corollary}
\newtheorem*{theorem*}{Theorem}
\newtheorem*{corollary*}{Corollary}
\theoremstyle{definition}
\newtheorem*{rem}{Remark}

\newcommand {\R}{\mathbb{R}} %% reals
\newcommand {\C}{\mathbb{C}} %% complex
\newcommand{\Hh}{\mathbb{H}}

\newcommand{\n}{\mathfrak{n}}
\renewcommand{\a}{\mathfrak{a}}
\renewcommand{\k}{\mathfrak{k}}
\newcommand{\g}{\mathfrak{g}}
\renewcommand{\sl}{\mathfrak{sl}}
\newcommand{\su}{\mathfrak{su}}
\newcommand{\p}{\mathfrak{p}}
\newcommand{\s}{\mathfrak{s}}
\renewcommand{\v}{\mathfrak{v}}
\newcommand{\z}{\mathfrak{z}}
\renewcommand{\Re}{{\mathrm{Re}}}

\DeclareMathOperator{\id}{id}
\DeclareMathOperator{\ad}{ad}
\DeclareMathOperator{\vol}{vol}

\DeclareMathOperator{\area}{area}

\DeclareMathOperator{\End}{End}

\DeclareMathOperator{\Tr}{Tr}

\DeclareMathOperator{\Ric}{Ric}

\begin{document}

\title [Codimension one foliation by Damek-Ricci spaces]{Minimal codimension one foliation of a symmetric space by Damek-Ricci spaces}

\author{Gerhard Knieper}
\author{John R. Parker}
\author{Norbert Peyerimhoff}
\address{Dept.\ of Mathematics, Ruhr University Bochum, 44780 Bochum, Germany}
\address{Dept.\ of Mathematical Sciences, Durham University, Durham DH1 3LE, UK}
\address{Dept.\ of Mathematical Sciences, Durham University, Durham DH1 3LE, UK}

\email{gerhard.knieper@rub.de}
\email{j.r.parker@durham.ac.uk}
\email{norbert.peyerimhoff@durham.ac.uk}

\date{\today}

\begin{abstract}
  In this article we consider solvable hypersurfaces of the form $N \exp(\R H)$
  with induced metrics in the symmetric space
  $M = SL(3,\C)/SU(3)$, where $H$ a suitable
  unit length vector in the subgroup $A$ of the Iwasawa decomposition
  $SL(3,\C) = NAK$. Since $M$ is rank $2$, $A$ is $2$-dimensional and we can parametrize these
  hypersurfaces via an angle $\alpha \in [0,\pi/2]$
  determining the direction of $H$. We show that one of the hypersurfaces
  (corresponding to $\alpha = 0$) is
minimally embedded and isometric to the non-symmetric $7$-dimensional 
Damek-Ricci space. We also provide an explicit formula for the Ricci curvatures
of these hypersurfaces and show that all hypersurfaces for $\alpha \in (0,\frac{\pi}{2}]$
admit planes of both negative and positive sectional curvature. Moreover, the symmetric space $M$ admits a minimal foliation 
with all leaves isometric to the non-symmetric $7$-dimensional 
Damek-Ricci space. 
\end{abstract}

\maketitle

\section{Introduction}

The purpose of this article is to study homogeneous hypersurfaces in  
the $8$-dimensional symmetric space $SL(3,\C)/SU(3)$. This rank two symmetric space can
be canonically identified with the solvable group $S = N A$ with
left invariant metric, using the Iwasawa decomposition 
$SL(3,\C) = N AK$, $K = SU(3)$. A specific orthonormal basis of
the associated two-dimensional Lie algebra $\a \subset T_eS$ 
is given by 
\begin{equation*}
H_0 = \begin{pmatrix} \frac{1}{2} & 0 & 0 \\ 0 & 0 & 0 \\
0 & 0 & -\frac{1}{2} \end{pmatrix}, \quad H_1 = \begin{pmatrix} 
\frac{1}{2\sqrt{3}} & 0 & 0 \\ 0 & - \frac{1}{\sqrt{3}} & 0 \\ 0 & 0 & \frac{1}{2\sqrt{3}} \end{pmatrix} 
\in \a. 
\end{equation*} 
Details are explained in Section \ref{sec:geom_props_complex} below. We have the following result:

\begin{theorem} \label{thm:main1}
  Let $S = N A$ be the symmetric space $SL(3,\C)/SU(3)$ with
  isometrically embedded hypersurfaces
  $S_H = N\exp(\R H)$, $H = \cos(\alpha) H_0 + \sin(\alpha) H_1$,
  $\alpha \in [0,\pi/2]$.
  
  Then $S_H$ is a simply connected constant mean curvature (CMC)
  hypersurface with mean curvature $- 4 \sin(\alpha)$. Moreover, 
  $S_{H_0}$ is minimally embedded in $S$ and isometric to the $7$-dimensional
  Damek-Ricci space. In particular, $S_{H_0}$ is a harmonic manifold, and therefore
  Einstein, with non-positive sectional curvature admitting planes of zero curvature.
  
 Moreover, the following are equivalent:
  \begin{itemize}
  \item[(a)] $S_H \subset S$ is minimally embedded;
  \item[(b)] the Cheeger constant of $S_H$ is maximal,
  \item[(c)] $H = H_0$.
  \end{itemize}
\end{theorem}

Damek-Ricci spaces are particularly important since they provide
counterexamples to the Lichnerowicz Conjecture. According to this conjecture, all
simply connected harmonic manifolds should be either flat or rank one symmetric spaces.
Harmonic manifolds are characerized by the property that all harmonic functions (i.e., $\Delta f = 0$)
have the mean value property, that is, the average of $f$ over any geodesic sphere agrees with the value of $f$
at the center (see \cite{Willm50}). It is well known that harmonic manifolds are Einstein (see \cite{Besse78}).
In the compact case, the Lichnerowicz Conjecture was settled affirmatively by  Szab\'{o} \cite{Sz90}.
It was shown by Knieper \cite{Knieper2012} that all non-flat non-positively curved
harmonic manifolds are Gromov hyperbolic and have the Anosov property. 
Damek-Ricci spaces are non-compact homogeneous harmonic manifolds of non-positive curvature and cover
all rank one symmetric spaces except for the real hyperbolic spaces.
It was shown by Heber \cite{Heber2006} that there are no other homogeneous harmonic manifolds than the ones mentioned above
and it is not known whether there are non-homogeneous harmonic examples. 
Dotti \cite{Dotti97} provided the first complete proof that Damek-Ricci spaces admit planes of vanishing curvature
if and only if they are non-symmetric. The smallest non-symmetric Damek-Ricci 
space has dimension $7$. In brief, the above theorem tells us that we can recover this $7$-dimensional
non-symmetric Damek-Ricci space as a minimal hypersurface of the specific
rank two symmetric space $SL(3,\C)/SU(3)$. For more information about Damek-Ricci spaces and recent results
on harmonic manifolds see, e.g., \cite{DR92} or the surveys \cite{BTV95,Rou2003,Knieper2016}.

\begin{rem}
There is an analogous result for homogeneous hypersurfaces in $SL(3,\R)/SO(3)$. The 
corresponding subspaces $S_H$ are then $4$-dimensional, simply connected CMC
 hypersurfaces with mean curvature $-2 \sin(\alpha)$ and $S_{H_0}$ is minimally embedded and isometric to the complex hyperbolic plane $\C H^2$.
 Since irreducible symmetric spaces do not admit totally geodesic hypersurfaces unless they have constant curvature (see \cite{Iw65} or, more generally
 \cite{BO2018}), note that there  is no totally geodesic embedding of $\C H^2$ into $SL(3,\R)/SO(3)$.
 \end{rem}

As a consequence of Theorem \ref{thm:main1} we obtain that $SL(3,\C)/SU(3)$ has a natural minimal codimension one foliation with leaves isometric to the 
$7$-dimensional Damek-Ricci space:

\begin{corollary} \label{cor:main2}
  Let $\alpha \in [0,\frac{\pi}{2}]$ and the flow $\{ \phi_H^s: S \to S \}_{s \in \R}$ be defined by
  $$ \phi_H^s(q) := q \cdot \exp(s T_H) $$
  with $T_H = \sin(\alpha) H_0  - \cos(\alpha) H_1 \bot T_eS_H$.
  Then $S$ admits a codimension one foliation with leaves $\{\phi_H^s(S_H) \}_{s \in \R}$. Moreover, the leaves
  of this foliation are pairwise equidistant and isometric to $S_H$. 
  
  In the particular case $\alpha = 0$, all leaves of this foliation are minimal and isometric to the Damek-Ricci space $S_{H_0}$, and $\phi_H^s$ is volume preserving both in $S$ and as a map between the leaves.
\end{corollary}

Finally, we investigate curvature properties of the hypersurfaces $S_H$. To state the result, we need a suitable orthonormal basis of $T_eS_H$, given
by $V,iV,W,iW,Z_0,iZ_0,H$ with
\begin{equation*} 
V = \begin{pmatrix} 0 & 1 & 0 \\ 0 & 0 & 0 \\ 0 & 0 & 0 \end{pmatrix}, \quad W = \begin{pmatrix} 0 & 0 & 0 \\ 0 & 0 & 1 \\ 0 & 0 & 0 \end{pmatrix}, \quad Z_0 = \begin{pmatrix} 0 & 0 & 1 \\ 0 & 0 & 0 \\ 0 & 0 & 0 
\end{pmatrix}. 
\end{equation*}

\begin{theorem}\label{thm:main3}
   Let $X = a V + b W + c Z_0 + t H \in T_eS_H$ with $a,b,c \in \C$ and $t \in \R$ be a unit vector, that is $|a|^2+ |b|^2 + |c|^2 + t^2 = 1$. Then the Ricci curvature of $X$ is given by
  \begin{multline*} 
%  \Ric^{S_H}(X) = \\ -3 + 2 \sin(\alpha) \left( \sqrt{3} (|b|^2-|a|^2) \cos(\alpha) + 
 % (|a|^2+|b|^2+2|c|^2)\sin(\alpha) \right).
  \Ric^{S_H}(X) = \\
  -3 + 4 \sin(\alpha)\left( \sin\left(\alpha-\frac{\pi}{3}\right)|a|^2 + \sin\left(\alpha+\frac{\pi}{3}\right)|b|^2
  + \sin(\alpha) |c|^2  \right).
  \end{multline*}
  In particular, the space $S_H$ has strictly negative Ricci curvature if and only if $\alpha \in
  [0,\frac{\pi}{3})$. $S_H$ admits directions of vanishing Ricci curvature for $\alpha = \frac{\pi}{3}$ 
  and directions of positive Ricci curvature for $\alpha \in (\frac{\pi}{3},\frac{\pi}{2}]$.
  In particular, $S_H$ is Einstein if and only if $\alpha = 0$.
  
  With regards to sectional curvature, the hypersurfaces $S_H$ have always planes of positive and negative curvature unless $\alpha = 0$. ($\alpha=0$ implies that $S_H$ is a non-positively curved Damek-Ricci space.) 
\end{theorem}

The structure of this article is as follows: In Section \ref{sec:geom_props_complex} we introduce the hypersurfaces $S_H$, compute their second fundamental form and
Cheeger constants. Section \ref{sec:ProofThm1.1} is devoted to the proof of Theorem \ref{thm:main1} and Corollary \ref{cor:main2}. The curvature results presented in
Theorem \ref{thm:main3} are proved in Section \ref{sec:proofThm1.3} using Maple computations. The Maple code can be found in Appendix \ref{app}. 

\bigskip

\noindent
{\bf Acknowledgement:} This research was partially supported by the program ''Research in Pairs'' of the MFO in 2019 and the
SFB/TR191 ''Symplectic structures in geometry, algebra and dynamics''. The authors are also grateful to Jens Heber to inform us about related results in \cite{Heber1998}.

\section{Basic geometric properties of the hypersurfaces $S_H$}
\label{sec:geom_props_complex}

\subsection{The Riemannian manifolds $S$ and $S_H$}

Henceforth, let $G = SL(3,\C)$ and $K = SU(3)$ and $\pi: G \to M = G/K$, $\pi(g) = g K$ be the canonical projection with $x_0=\pi(e)$. 

We briefly recall the construction of a Riemannian metric which makes $M = G/K$ 
a symmetric space: A Cartan involution on $\g$ is given by $\theta: \g \to \g$, 
$\theta(X) = -\bar X^\top$. The Killing form 
$$ B(X_1,X_2) = \Tr (\ad X_1 \circ \ad X_2) = 12\, \Re \Tr (X_1 X_2), $$
gives rise to the following inner product on $\g$:
\begin{eqnarray} \label{eq:inprodg}
\langle X_1,X_2 \rangle_\g = - \frac{1}{6} B(X_1,\theta X_2) &=& 2 \Re \Tr (X_1 \bar X_2^\top ) \\
&=& 2 \Re \left( \sum_{i,j} (X_1)_{ij} \overline{(X_2)_{ij}} \right). \nonumber
\end{eqnarray}
Since $\ker D\pi(e) = \k$, the differential $D\pi(e)$ provides a canonical identification of $\p$ and $T_{x_0}M$, where 
\begin{equation} \label{eq:cartan}
\g = \p \oplus \k, \quad X \mapsto \frac{1}{2}(X-\theta(X)) + \frac{1}{2}(X+\theta(X)) 
\end{equation}
is the Cartan decomposition with $\p = \{ X \in \g \mid \theta(X) = - X \}$ and $\k =  \{ X \in \g \mid \theta(X) = X \}$. The restriction
of $\langle \cdot,\cdot \rangle_\g$ to $\p$ induces an inner product on $T_{x_0}M$.
Left-translation induces a Riemannian metric on $M$ such that $M$ becomes a rank two symmetric space of non-compact type. 

Alternatively, we can view $M=G/K$ as a solvable group $S$ with left invariant metric: the Iwasawa decomposition $\g = \sl(3,\C) = \n \oplus \a \oplus \k$ on the 
Lie algebra level is given by
\begin{eqnarray}
\n &=& \left\{ \begin{pmatrix} 0 & a & c \\ 0 & 0 & b \\ 0 & 0 & 0 \end{pmatrix} \Bigg\vert\, a,b,c \in \C \right\}, 
\label{eq:Lielagn} \\
\a &=& \left\{ \begin{pmatrix} t_1 & 0 & 0 \\ 0 & t_2 & 0 \\ 0 & 0 & t_3 \end{pmatrix} \Bigg\vert\, t_1,t_2,t_3 \in \R, t_1+t_2+t_3=0 \right\}, \nonumber \\
\k &=& \su(3) = \{ X \in \ \sl(3,\C) \mid X = - \bar X^\top \}. \nonumber
\end{eqnarray}
Let $N,A \subset G$ be the Lie groups corresponding to $\n$ and $\a$. Then the restriction of $\pi: G \to M$ to the solvable group $S = NA$ defines
a diffeomorphism $\pi\vert_{S}: S \to M$, $s \mapsto s K$. The pull-back of the Riemannian metric on $M$ via this diffeomorphism equips $S$ with a left-invariant metric. 
This left-invariant metric induces an inner product $\langle \cdot, \cdot \rangle_\s$ on the Lie algebra$\s = \n \oplus \a = T_eS$ of $S$. 
Using \eqref{eq:cartan} we have the following identifications:
\begin{eqnarray*} 
T_eS = \s \to &T_{x_0}M = \g / \k& \to \p, \\
\quad X \mapsto &X + \k& \mapsto \frac{1}{2}(X - \theta(X)) = \frac{1}{2}(X + \bar X^\top),
\end{eqnarray*}
leading to the linear isometry $\phi: \s \to \p$, $\phi(X) = \frac{1}{2}(X+\bar X^\top)$.
Our next aim is to calculate $\langle \cdot,\cdot \rangle_\s$: Let $X_1 = Y_1 + H, X_2 = Y_2 + \tilde H \in \s$ with $Y_1,Y_1 \in \n$ and $H,\tilde H \in \a$. We have
$$ \langle X_1, X_2 \rangle_\s = \langle \phi(X_1), \phi(X_2) \rangle_\g =  \frac{1}{4} \langle X_1 + \bar X_1^\top, X_2 + \bar X_2^\top \rangle_\g. $$
Using \eqref{eq:inprodg}, we obtain $\a \bot \n$ and $\n \bot \n^\top$ with respect to
$\langle \cdot, \cdot \rangle_\g$ and, therefore, we have
\begin{eqnarray} \label{eq:inprods}
\langle Y_1+H, Y_2+\tilde H \rangle_\s &=& \frac{1}{4} \langle Y_1 + \bar Y_1^\top + 2H, Y_2 + \bar Y_2^\top +2\tilde H \rangle_\g \\
&=&  \frac{1}{2} \langle Y_1,Y_2 \rangle_\g + \langle H,\tilde H \rangle_\g \nonumber \\ &=& \Re\left(\Tr(Y_1 \bar Y_2^\top) \right)+ 2 \Tr(H \tilde H). \nonumber 
\end{eqnarray}
In particular, we have $\a \bot \n$ with respect to $\langle \cdot, \cdot \rangle_\s$ and the matrices 
\begin{equation} \label{eq:H0H1} 
H_0 = \begin{pmatrix} \frac{1}{2} & 0& 0 \\ 0 & 0 & 0 \\ 0 & 0 & -\frac{1}{2} \end{pmatrix} \quad \text{and} \quad H_1 = \begin{pmatrix} \frac{1}{2\sqrt{3}} & 0 & 0 \\ 0 & - \frac{1}{\sqrt{3}} & 0 \\ 0 & 0 & \frac{1}{2\sqrt{3}} \end{pmatrix} 
\end{equation}
form an orthonormal basis of the $2$-dimensional vector space $\a$.
Any matrix in $\a$ of unit length can then be expressed as 
\begin{equation} \label{eq:H}
H = \cos(\alpha) H_0 + \sin(\alpha) H_1,
\end{equation}
and we define the corresponding hypersurface by 
$$ S_H = N \exp(\R H) \subset S = NA. $$ 
%We denote the set of all hypersurfaces of this kind by $\F$. 

\subsection{The second fundamental form of $S_H$}

Next we want to compute the second fundamental form of $S_H \subset S$ explicitly. The vector $T_H = \sin(\alpha) H_0 - \cos(\alpha) H_1 \in \s$ is a unit vector orthogonal to $T_eS_H = 
\s_H = \n \oplus \R H$ with respect to $\langle \cdot, \cdot \rangle_\s$. Its left invariant extension along $S_H$ provides a global unit normal vector field of $S_H \subset S$.
Any $X \in \s_H$ can be written as $X = a V + b W + c Z_0 + t H$ with $a,b,c \in \C$, $t \in \R$ and
\begin{equation} \label{eq:VWZ0}
V = \begin{pmatrix} 0 & 1 & 0 \\ 0 & 0 & 0 \\ 0 & 0 & 0 \end{pmatrix}, \quad W = \begin{pmatrix} 0 & 0 & 0 \\ 0 & 0 & 1 \\ 0 & 0 & 0 \end{pmatrix}, \quad Z_0 = \begin{pmatrix} 0 & 0 & 1 \\ 0 & 0 & 0 \\ 0 & 0 & 0 
\end{pmatrix}. 
\end{equation}
It is easy to see that $V,iV,W,iW,Z_0,iZ_0,H,T_H$ form an orthonormal basis of $\s$. Henceforth $\nabla^S$ denotes the Levi-Civita connection of $S$.

\begin{prop} \label{prop:2FF}
  Let $H = \cos(\alpha) H_0 + \sin(\alpha) H_1$ with $H_0, H_1$ given in \eqref{eq:H0H1}. Then the
  second fundamental form of $S_H$ is given by
  \begin{multline*}
  \nabla_{a V + b W + c Z_0 + t H}^S T_H = \\
  a \left( \frac{\sqrt{3}}{2} \cos \alpha - \frac{\sin \alpha}{2} \right) V + b \left( -\frac{\sqrt{3}}{2} \cos \alpha - \frac{\sin \alpha}{2} \right) W - c\, (\sin \alpha) Z_0. 
  \end{multline*}
  Moreover $S_H$ is a CMC hypersurface in $S$ with mean curvature 
  \begin{equation} \label{eq:meancurv}
  M(\alpha) = - 4 \sin \alpha.
  \end{equation}
\end{prop}

\begin{rem} Note that the hypersurfaces $S_H$ are horospheres iff $\alpha \in [\frac{\pi}{3},\frac{\pi}{2}]$ (with the singular horosphere at $\alpha = \frac{\pi}{3}$
and the barycentric horosphere at $\alpha = \frac{\pi}{2}$) in which case the eigenvalues of the second fundamental form, given by 
$\pm \frac{\sqrt{3}}{2} \cos \alpha - \frac{\sin \alpha}{2}, - \sin(\alpha), 0$ are non-positive. 
\end{rem}

\begin{proof}
  Using the canonical identification of $\s$ with left invariant vector fields on $S$ and applying Koszul's formula, we obtain
  $$ \langle \nabla_{X_1}^S X_2, X_3 \rangle_\s = \frac{1}{2} \left( \langle X_1, [X_3,X_2] \rangle_\s + \langle X_2,[X_1,X_3] \rangle_\s + \langle  X_3,[X_1,X_2] \rangle_\s \right) $$
  for $X_1,X_2,X_3 \in \s$. This in particular implies,
  \begin{equation} \label{eq:2FF_Gauss}  
  \langle \nabla_{X_1}^S T_H, X_2 \rangle_\s = \frac{1}{2} \left( \langle X_1, [X_2,T_H] \rangle_\s + \langle  X_2,[X_1,T_H] \rangle_\s \right), 
  \end{equation}
  since $\langle T_H, [X_1,X_2] \rangle_\s = 0$ because of $[X_1,X_2] \in \s_H$. 
  A straightforward calculation shows
  \begin{equation} \label{eq:Liebracket}
  [E_{ij},T] = E_{ij}T - TE_{ij} = (t_j-t_i)E_{ij}
  \end{equation}
   with $E_{ij}$ a $3 \times 3$ matrix with all entries equals $0$ except for one entry equals $1$ at position $(i,j)$ and $T$ a diagonal
  matrix with diagonal entries $(t_1,t_2,t_3)$. Since
  $$ T_H = \begin{pmatrix} \frac{\sin \alpha}{2} - \frac{\cos \alpha}{2 \sqrt{3}} & 0 & 0 \\ 0 & \frac{\cos \alpha}{\sqrt{3}} & 0 \\ 0 & 0 & -\frac{\sin \alpha}{2} - \frac{\cos \alpha}{2 \sqrt{3}} \end{pmatrix}, $$
  this implies that
  \begin{multline*}
  [a V + b W + c Z_0 + t H,T_H] = \\ a\left( -\frac{1}{2}\sin \alpha + \frac{\sqrt{3}}{2}\cos \alpha \right)V
  + b\left( -\frac{1}{2}\sin \alpha - \frac{\sqrt{3}}{2}\cos \alpha \right)W 
  - c (\sin \alpha) Z_0.
  \end{multline*}
%  \begin{eqnarray*}
%  [a V + b W + c Z + t H,T_H] &=& a\left( -\frac{1}{2}\sin \alpha + \frac{\sqrt{3}}{2}\cos \alpha \right)V \\
%  &+& b\left( -\frac{1}{2}\sin \alpha - \frac{\sqrt{3}}{2}\cos \alpha \right)W \\
%  &-& c (\sin \alpha) Z.
%  \end{eqnarray*}
  Consequently, $\nabla_\bullet T_H$ has diagonal structure with respect to $V,iV,W,iW,$ $Z_0,iZ_0,H,$ and we have
  $$ \nabla_V^S T_H = \langle \nabla_V^S T_H, V \rangle_\s V = \langle V, [V,T_H] \rangle_\s V = \left( -\frac{1}{2} \sin \alpha + \frac{\sqrt{3}}{2}\cos \alpha \right) V, $$
  and similarly for the other unit vectors. This finishes the proof of Proposition \ref{prop:2FF}.
\end{proof}

\subsection{The Cheeger constant of $S_H$}

The Cheeger isoperimetric constant $h(M)$ of a complete non-compact Riemannian
manifold $M$ is defined by
$$ h(M) = \inf_{K \subset M} \frac{\area(\partial K)}{\vol(K )}, $$
where $K$ ranges over all connected, open submanifolds of $M$ with compact closure and smooth boundary. 

A formula
for this constant was given in \cite{PS2004} for general solvable groups with left invariant metric. Since $S_H$ is a solvable group,
we obtain from this formula
\begin{equation} \label{eq:CheegSH}
h(S_H) = \max_{X \in \s_H, \Vert X \Vert_\s = 1} \Tr(\ad X), 
\end{equation}
where $\ad X(\tilde X) = [X,\tilde X]$ is viewed as linear transformation on the $7$-dimensional real vector space $\s_H$
spanned by $V,iV,W,iW,Z_0,iZ_0,H$. This is the main ingredient of the proof of the following result: 

\begin{prop}  \label{prop:Cheeger} 
Let $H = \cos(\alpha) H_0 + \sin(\alpha) H_1$ with $H_0, H_1$ given in \eqref{eq:H0H1}. Then the
  Cheeger constant of $S_H$ is given by
  $$ h(S_H) = 4 \cos \alpha. $$
  In particular, $S_{H_1}$ has a vanishing Cheeger constant.
\end{prop}

\begin{proof}
  In view of \eqref{eq:CheegSH} we only have to calculate $\Tr(\ad X)$ for $X = a V + b W + c Z_0 + t H$ with
  $|a|^2+|b|^2+|c|^2+t^2 = 1$ with $a,b,c \in \C$ and $t \in \R$. Using \eqref{eq:Liebracket} we conclude
  for $e \in \{1,i\}$ that
  \begin{eqnarray*}
  [H,e V] &=& \left( \frac{\cos \alpha}{2} + \frac{\sqrt{3}}{2} \sin \alpha \right) e V, \\
  {[H,e W]} &=& \left( \frac{\cos \alpha}{2} - \frac{\sqrt{3}}{2} \sin \alpha \right) e W, \\
  {[H,e Z_0]} &=& (\cos \alpha) e Z_0.
  \end{eqnarray*}
  Note that the traces of $\ad eV, \ad eW$ and $\ad eZ_0$ vanish since the matrix representations of these
  operators  have zero for each diagonal entry. This implies that
  \begin{multline*} 
  h(S_H) = \max_{X \in \s_H, \Vert X \Vert_\s = 1} \Tr(\ad X) = \Tr(\ad H) = \\
  2 \left( \frac{\cos \alpha}{2} + \frac{\sqrt{3}}{2} \sin \alpha \right) + 2 \left( \frac{\cos \alpha}{2} - \frac{\sqrt{3}}{2} \sin \alpha \right)
  + 2 \cos \alpha = 4 \cos \alpha.
  \end{multline*}
\end{proof}

\section{Proof of Theorem \ref{thm:main1} and Corollary \ref{cor:main2}}
\label{sec:ProofThm1.1}

For the reader's convenience, we recall Theorem \ref{thm:main1} from the Introduction:

\begin{theorem*}
  Let $S = N A$ be the symmetric space $SL(3,\C)/SU(3)$ with
  isometrically embedded hypersurfaces
  $S_H = N\exp(\R H)$, $H = \cos(\alpha) H_0 + \sin(\alpha) H_1$,
  $\alpha \in [0,\pi/2]$, with $H_0, H_1$ given in \eqref{eq:H0H1}.
  
  Then $S_H$ is a simply connected CMC
  hypersurface with mean curvature $- 4 \sin(\alpha)$ and  
  $S_{H_0}$ is minimally embedded in $S$ and isometric to the $7$-dimensional
  Damek-Ricci space. In particular, $S_{H_0}$ is a harmonic manifold, and therefore
  Einstein, with non-positive sectional curvature admitting planes of zero curvature.
  
 Moreover, the following are equivalent:
  \begin{itemize}
  \item[(a)] $S_H \subset S$ is minimally embedded;
  \item[(b)] the Cheeger constant of $S_H$ is maximal,
  \item[(c)] $H = H_0$.
  \end{itemize}
\end{theorem*}

\begin{proof}
  The solvable group $S_{H_0}$ with left invariant metric is a Damek-Ricci space if the
  following properties of $(\s_{H_0},\langle \cdot, \cdot \rangle_\s)$ are satisfied:
  \begin{itemize}
  \item[(1)] $\s_{H_0} = \n \oplus \R H_0$, $\n \bot H_0$ and $H_0$ is a unit vector
  with respect to $\langle \cdot, \cdot \rangle_\s$;
  \item[(2)] $\n = \v \oplus \z$ with $[\v,\v] \subset \z$ and $[\v,\z],[\z,\z]=\{0\}$
  (that is $\n$ is $2$-step nilpotent);
  \item[(3)] $\v \bot \z$ with respect to $\langle \cdot, \cdot \rangle_\s$;
  \item[(4)] let $Z \in \z$; then the map $J_Z \in \End(\v)$, defined by
  $$
  \langle J_Z(U_1),U_2 \rangle_\s = \langle Z,[U_1,U_2] \rangle_\s \quad \text{for all $U_1,U_2 \in \v$}, 
  $$
  satisfies $J_Z^2 = - \Vert Z \Vert^2 \id_\v$; 
  \item[(5)] $[H_0,U] = \frac{1}{2} U$ for all $U \in \v$ and $[H_0,Z] = Z$ for all $Z \in \z$.
  \end{itemize}
  We note that a Lie algeba $\n$ satisfying properties (2), (3) and (4) is called a Lie algebra of Heisenberg type.
  
  Properties (1), (2), (3) and (5) are obviously satisfied by choosing $\v = \C V \oplus \C W$ and $\z = \C Z_0$
  since $V,iV,W,iW,Z_0,iZ_0,H_0$ are an orthonormal basis of $\s_{H_0}$ with respect to $\langle \cdot, \cdot \rangle_\s$.
  For example, (2) follows from $[V,W] = VW - WV = Z_0$ and (5) follows from
  \begin{multline*}
  \left[ H_0,\begin{pmatrix} 0 & a & c \\ 0 & 0 & b \\ 0 & 0 & 0 \end{pmatrix} \right] = \\
  \frac{1}{2} \begin{pmatrix} 1 & 0 & 0 \\ 0 & 0 & 0 \\ 0 & 0 & -1  \end{pmatrix} \begin{pmatrix} 0 & a & c \\ 0 & 0 & b \\ 0 & 0 & 0 \end{pmatrix} -
  \frac{1}{2} \begin{pmatrix} 0 & a & c \\ 0 & 0 & b \\ 0 & 0 & 0 \end{pmatrix}  \begin{pmatrix} 1 & 0 & 0 \\ 0 & 0 & 0 \\ 0 & 0 & -1  \end{pmatrix}
  = \begin{pmatrix} 0 & \frac{a}{2} & c \\ 0 & 0 & \frac{b}{2} \\ 0 & 0 & 0 \end{pmatrix}.
  \end{multline*}
  To show (4), we define for $Z = z Z_0$, $ z \in \C$,
  $$ J_Z \begin{pmatrix} 0 & a & 0 \\ 0 & 0 & b \\ 0 & 0 & 0 \end{pmatrix} = z \begin{pmatrix} 0 & -\bar b & 0 \\ 0 & 0 & \bar a \\ 0 & 0 & 0 \end{pmatrix}. $$
  Then we have
  \begin{multline*} 
  \left \langle J_{Z_0} \begin{pmatrix} 0 & a & 0 \\ 0 & 0 & b \\ 0 & 0 & 0 \end{pmatrix}, \begin{pmatrix} 0 & c & 0 \\ 0 & 0 & d \\ 0 & 0 & 0 \end{pmatrix}
  \right \rangle_\s = \Re(\overline{ad-bc}) = \\ \left \langle \begin{pmatrix} 0 & 0 & 1 \\ 0 & 0 & 0 \\ 0 & 0 & 0 \end{pmatrix}, \begin{pmatrix} 0 &  0 & ad-bc \\ 0 & 0 & 0 \\ 0 & 0 & 0 \end{pmatrix} \right \rangle_\s
  = \left \langle Z_0, \left [\begin{pmatrix} 0 & a & 0 \\ 0 & 0 & b \\ 0 & 0 & 0 \end{pmatrix},\begin{pmatrix} 0 & c & 0 \\ 0 & 0 & d \\ 0 & 0 & 0 \end{pmatrix} \right] \right \rangle_\s
  \end{multline*}
  and
  $$ J_{Z_0}^2 \begin{pmatrix} 0 & a & 0 \\ 0 & 0 & b \\ 0 & 0 & 0 \end{pmatrix} = J_{Z_0} \begin{pmatrix} 0 & -\bar b & 0 \\ 0 & 0 & \bar a \\ 0 & 0 & 0 \end{pmatrix} =
  - \begin{pmatrix} 0 & a & 0 \\ 0 & 0 & b \\ 0 & 0 & 0 \end{pmatrix}. $$
  
  This shows that $S_{H_0}$ is the $7$-dimensional Damek-Ricci space which is, therefore, a harmonic manifold (see \cite{DR92}). The space $S_{H_0}$ cannot be a 
  symmetric space since $\dim_\R \z = 2$ and the centres
  of symmetric Damek-Ricci spaces must have dimension $1, 3$ or $7$. It was shown independently by \cite{Bo85} and \cite{Damek2.87} that all Damek-Ricci
  spaces have non-positive sectional curvature and by \cite{Dotti97} that these spaces admit planes of zero curvature if and only if they are non-symmetric.
  
  Finally, the equivalences of (a), (b) and (c) follow immediately from Propositions \ref{prop:2FF} and \ref{prop:Cheeger}.
\end{proof}

\begin{rem}
In the case of the rank two symmetric space $\Hh^2 \times \Hh^2$ (where $\Hh^k$ denotes the $k$-dimensional real hyperbolic space) a similar analysis shows that $S_{H_0}$ is of constant negative curvature, that is, agrees with $\Hh^3$ up to scaling. Here the direction $H_0$ in the flat $\a$ is characterized by the fact that $S_{H_0}$ is minimally embedded in $\Hh^2 \times \Hh^2$. It would be interesting to investigate which of the corresponding hypersurfaces in rank two symmetric spaces of non-compact type are harmonic manifolds.
\end{rem}

Theorem \ref{thm:main1} has the following consequence:

\begin{corollary*}
  Let $\alpha \in [0,\frac{\pi}{2}]$ and the flow $\{ \phi_H^s: S \to S \}_{s \in \R}$ be defined by
  $$ \phi_H^s(q) := q \cdot \exp(s T_H). $$
  Then $S$ admits a codimension one foliation with leaves $\{\phi_H^s(S_H) \}_{s \in \R}$. Moreover, the leaves
  of this foliation are pairwise equidistant and isometric to $S_H$. 
  
  In the particular case $\alpha = 0$, all leaves of this foliation are minimal and isometric to the Damek-Ricci space $S_{H_0}$, and $\phi_H^s$ is volume preserving both in $S$ and as a map between the leaves.
\end{corollary*}

\begin{proof}
  By abuse of notation, we extend $T_H \in \s = T_eS$ to a global unit vector field on $S$, again denoted by $T_H$, orthogonal to $S_H$ and given by
  $$ T_H(q) = \frac{d}{ds}\Big\vert_{s=0} q \exp(s T_H). $$
  Then $\phi_H^s$ is the associated flow and its flow lines $s \mapsto \phi_H^s(q)$ are geodesics
  in $S$ through $q$. This implies that the leaves are equidistant.
  
  Next we show that all leaves are isometric to $S_H$: Let $F_H^s: S \to S$ be the isometry
  $F_H^s(q) = \exp(s T_H) q$. Then we have for all $q \in S_H$ that there exists $q' \in S_H$
  with
  \begin{equation} \label{eq:flow_isom} 
  \phi_H^s(q) = F_H^s(q'), 
  \end{equation}
  and, therefore, $\phi_H^s(S_H)$ and $F_H^s(S_H)$ coincide as sets and are isometric
  to $S_H$. Indeed, if
  $$ q = \begin{pmatrix} 1 & x & z \\ 0 & 1 & y \\ 0 & 0 & 1 \end{pmatrix} \exp(t H) \in S_H$$
  and
  $$ \exp(s T_H) =  \begin{pmatrix} e^{\tau_1} & 0 & 0 \\ 0 & e^{\tau_2} & 0 \\ 0 & 0 & e^{\tau_3}\end{pmatrix},
  $$
  with suitable $\tau_1,\tau_2,\tau_3 \in \R$, then \eqref{eq:flow_isom} is satisfied if
  $$ q' = \begin{pmatrix} 1 & e^{\tau_2-\tau_1} x & e^{\tau_3-\tau_1} z \\ 0 & 1 & e^{\tau_3-\tau_2} y \\ 0 & 0 & 1 \end{pmatrix} \exp(t H) \in S_H. $$
  
  We know from Theorem \ref{thm:main1} that $S_{H_0}$ is a Damek-Ricci space and minimal in $S$. Since $F_H^s$ is an isometry mapping leaves to leaves, the mean curvature is preserved for all leaves. Finally, the volume distortion of the flow $\phi_H^s$ on both $S$ and as a map between the leaves is given by $e^{s M(\alpha)}$ with the mean curvature $M(\alpha) = -4 \sin(\alpha)$ given in \eqref{eq:meancurv}. Hence $\phi_H^s$ is volume preserving for $\alpha = 0$.
  \end{proof}

\section{Curvature considerations for the hypersurfaces $S_H$}
\label{sec:proofThm1.3}

This section is devoted to the proof of Theorem \ref{thm:main3} from the Introduction which states the following:

\begin{theorem*}
  Let $X = a V + b W + c Z_0 + t H \in \s_H$ with $V,W,Z_0$ given in \eqref{eq:VWZ0} and $H$ given in \eqref{eq:H}.   We assume that
  $X$ is a unit vector, that is
   $a,b,c \in \C$ and $t \in \R$ with $|a|^2+ |b|^2 + |c|^2 + t^2 = 1$. Then the Ricci curvature of $X$ is given by
  \begin{multline*} 
%  \Ric^{S_H}(X) = \\ -3 + 2 \sin(\alpha) \left( \sqrt{3} (|b|^2-|a|^2) \cos(\alpha) + 
 % (|a|^2+|b|^2+2|c|^2)\sin(\alpha) \right).
  \Ric^{S_H}(X) = \\
  -3 + 4 \sin(\alpha)\left( \sin\left(\alpha-\frac{\pi}{3}\right)|a|^2 + \sin\left(\alpha+\frac{\pi}{3}\right)|b|^2
  + \sin(\alpha) |c|^2  \right).
  \end{multline*}
  In particular, the space $S_H$ has strictly negative Ricci curvature if and only if $\alpha \in
  [0,\frac{\pi}{3})$. $S_H$ admits directions of vanishing Ricci curvature for $\alpha = \frac{\pi}{3}$ 
  and directions of positive Ricci curvature for $\alpha \in (\frac{\pi}{3},\frac{\pi}{2}]$.
  In particular, $S_H$ is Einstein if and only if $\alpha = 0$.
  
  With regards to sectional curvature, the hypersurfaces $S_H$ have always planes of positive and negative curvature unless $\alpha = 0$. ($\alpha=0$ implies that $S_H$ is a non-positively curved Damek-Ricci space.) 
\end{theorem*}

Before we enter the proof we like to make the following general remark.

\begin{rem}
  The following result was shown in Heber \cite[Theorem 4.18]{Heber1998} (related to earlier work by Wolter \cite{Wolter1991}): Let $\s = \a \oplus \n$ be a Lie algebra of Iwasawa type with inner product $Q$ which is Einstein and
  $H_Q \in \s$ be the vector defined by $Q(H_Q,X) = \Tr \ad_X$ for all $X \in \s$. Then the metric subalgebra
  $(\a' \oplus \n,Q)$ with non-trivial subspace $\a' \subset \a$ is Einstein if and only if $H_Q \in \a'$. In particular,
  $(\R H_Q \oplus \n,Q)$ is Einstein.
  
  Note that our Lie algebra $(\s,\langle \cdot,\cdot \rangle_s)$ is Einstein since its corresponding Lie group with
  left invariant metric is a symmetric space and we can apply this result with $Q = \langle \cdot,\cdot \rangle_\s$.
  A straightforward calculation yields then $H_Q = 4 \cdot H_0$ and Heber's result agrees with our result that
  amongst all hypersurfaces $S_H$ with $H = \cos(\alpha)H_0 + \sin(\alpha)H_1$ only $S_{H_0}$ is an Einstein manifold. 
  
  It would be interesting to investigate which of the homogeneous Einstein manifolds appearing in the more general setting of Heber are Damek-Ricci spaces. 
\end{rem}

\begin{proof} 
Let $R^S$ be the Riemannian curvature tensor of $S$ given by
$$ R^S(X_1,X_2)X_3 = \nabla_{X_1}^S \nabla_{X_2}^S X_3 - \nabla_{X_2}^S \nabla_{X_1}^S X_3 - \nabla_{[X_1,X_2]}^S X_3 $$
and $R^{S_H}$ be the corresponding curvature tensor of $S_H$.

The derivation of the expression \eqref{eq:ricX} is based on the Gauss equation:
 \begin{multline} \label{eq:gauss_eq}
\langle R^{S_H}(X_1,X)X,X_1 \rangle_\s = \langle R^S(X_1,X)X,X_1 \rangle_\s + \\ \langle \nabla_{X_1}^S T_H, X_1 \rangle_\s \langle \nabla_{X}^S T_H,X \rangle_\s - \left( \langle \nabla_{X_1}^S T_H,X \rangle_\s \right)^2,
\end{multline}
where $X_1 \in \s_H = \n \oplus \R H = T_e S_H$ and $T_H = \sin(\alpha) H_0 - \cos(\alpha) H_1 \in \s$. The ingredients in \eqref{eq:gauss_eq} are explicitly calculated using
$$ \langle R^S(X_1,X_2)X_2,X_1 \rangle_\s = - \langle [[\phi(X_1),\phi(X_2)],\phi(X_2)],\phi(X_1) \rangle_\g
$$
from the theory of symmetric spaces (see, e.g., \cite[Theorem IV.4.2]{Helgason62}) and the following consequence of Koszul's formula (see \eqref{eq:2FF_Gauss}):
$$
\langle \nabla_{X_1}^S T_H, X_2 \rangle_\s = \frac{1}{2} \left( \langle \Phi(X_1), \Phi([X_2,T_H]) \rangle_\g + \langle  \Phi(X_2),\Phi([X_1,T_H]) \rangle_\g \right).
$$
The Ricci curvature is then given by
\begin{multline} \label{eq:Ric_SH}
\Ric^{S_H}(X) = \langle R^{S_H}(V,X)X,V \rangle_\s + \langle R^{S_H}(iV,X)X,iV \rangle_\s + \\\langle R^{S_H}(W,X)X,W \rangle_\s + \langle R^{S_H}(iW,X)X,iW \rangle_\s + \\ \langle R^{S_H}(Z_0,X)X,Z_0 \rangle_\s + \langle R^{S_H}(iZ_0,X)X,iZ_0 \rangle_\s + \langle R^{S_H}(H,X)X,H \rangle_\s. 
\end{multline}

The calculation of \eqref{eq:Ric_SH} in the case $X = a V + b W + C Z_0 + t H$ with $|a|^2+|b|^2+|c|^2+t^2=1$ was done with Maple (see Appendix \ref{app}) with the following result:
\begin{multline*} 
\Ric^{S_H}(X) = - 2 \sqrt{3} \sin(\alpha)\cos(\alpha) (|a|^2-|b|^2) \\ - 2 (|a|^2+|b|^2+2|c|^2)
\cos^2(\alpha) - 3t^2 - |a|^2 - |b|^2 + |c|^2,
\end{multline*}
which simplifies to
\begin{multline}  \label{eq:ricX} 
\Ric^{S_H}(X) = \\ -3 + 4 \sin(\alpha)\left( \sin\left(\alpha-\frac{\pi}{3}\right)|a|^2 + \sin\left(\alpha+\frac{\pi}{3}\right)|b|^2
  + \sin(\alpha) |c|^2  \right),
\end{multline}
using $|a|^2+|b|^2+|c|^2=1-t^2$.

In order to find the maximum of \eqref{eq:ricX} for a given value of $\alpha \in [0,\frac{\pi}{2}]$, it is sufficient to 
assume that $a,b,c$ are real with $a^2+b^2+c^2 \le 1$. Let
$$ f_\alpha(a,b,c) = 4 \sin(\alpha)\left( \sin\left(\alpha-\frac{\pi}{3}\right)|a|^2 + \sin\left(\alpha+\frac{\pi}{3}\right)|b|^2
  + \sin(\alpha) |c|^2  \right). $$
Since $f_\alpha(a,b,c)$ is a homogeneous polynomial of degree $2$, we have
$$ \max_{a^2+b^2+c^2 \le 1} f_\alpha(a,b,c) = \max_{a^2+b^2+c^2 = 1} f_\alpha(a,b,c). $$
When $a^2+b^2+c^2 = 1$, it is obvious that the maximal value of $f_\alpha$ is equal to
$$
4\sin(\alpha) \max\left\{ \sin\left(\alpha-\frac{\pi}{3}\right), \sin\left(\alpha+\frac{\pi}{3}\right), \sin(\alpha) \right\},
$$
and we obtain
$$ -3 + \max _{a^2+b^2+c^2 = 1} f_\alpha(a,b,c) = \begin{cases} 
4 \sin(\alpha) \sin(\alpha+\frac{\pi}{3}) - 3 & \text{if $0 \le \alpha \le \frac{\pi}{3}$;} \\
4 \sin^2(\alpha) - 3 & \text{if $\frac{\pi}{3} < \alpha \le \frac{\pi}{2}$.} \end{cases}
$$
This means that the maximum is strictly monotone in $\alpha$ and vanishes at $\alpha = \frac{\pi}{3}$, which implies the statements about the Ricci curvature signs. 

Finally, we have $f_0(a,b,c) = -3$ and $S_H$ is Einstein for $\alpha = 0$. For $\alpha \in
(0,\frac{\pi}{2}]$, we have $f_\alpha(0,0,c) = 4 c^2 \sin^2(\alpha) - 3$ which is non-constant since $c \in [-1,1]$. This implies that $S_H$ is not Einstein in this case.

Concerning sectional curvature, we consider the plane $\sigma \subset \s_H$ 
spanned by the orthonormal vectors
$$ X_1 = \sqrt{\frac{2}{3}} W + \frac{1}{\sqrt{3}} Z_0 \quad \text{and} \quad
X_2 = - \sqrt{\frac{2}{3}} i W + \frac{1}{\sqrt{3}} i Z_0. $$
Using \eqref{eq:gauss_eq} we obtain again with the help of Maple (see Appendix \ref{app})
$$ K^{S_H}(\sigma) = \langle R^{S_H}(X1,X2)X2,X1 \rangle_\s =
\frac{4}{3\sqrt{3}} \sin(\alpha)\cos(\alpha) + \frac{1}{9} \sin(\alpha)^2. $$
This expression vanishes only if $\alpha = 0$ and is strictly positive for any $\alpha \in
(0,\frac{\pi}{2}]$. Moreover, since $\Ric^{S_H}(H) = -3$ for all $\alpha \in [0,\frac{\pi}{2}]$, there
are also planes of strictly negative curvature.
\end{proof}

\appendix

\section{Maple Calculations}
\label{app}

\DefineParaStyle{Maple Bullet Item}
\DefineParaStyle{Maple Heading 1}
\DefineParaStyle{Maple Warning}
\DefineParaStyle{Maple Heading 4}
\DefineParaStyle{Maple Heading 2}
\DefineParaStyle{Maple Heading 3}
\DefineParaStyle{Maple Dash Item}
\DefineParaStyle{Maple Error}
\DefineParaStyle{Maple Title}
\DefineParaStyle{Maple Text Output}
\DefineParaStyle{Maple Normal}
\DefineCharStyle{Maple 2D Output}
\DefineCharStyle{Maple 2D Input}
\DefineCharStyle{Maple Maple Input}
\DefineCharStyle{Maple 2D Math}
\DefineCharStyle{Maple Hyperlink}

In this appendix, we discuss the Maple code for the calculation of Ricci curvature of hypersurface $S_H$ within $SL(3,C)/SU(3)$ and the existence of planes with positive sectional curvatures.
%\noindent
%\begin{maplegroup}
%\begin{Maple Normal}{
%Calculation of Ricci curvature of hypersurface $S_H$ within $SL(3,C)/SU(3)$ and the existence of planes with positive sectional curvatures.}\end{Maple Normal}
\bigskip

The following lines guarantee that Maple treats $\alpha$ and $t$ as real variables:
\bigskip

\begin{maplegroup}
\begin{mapleinput}
\mapleinline{inert}{1d}{with(LinearAlgebra):}{\[\displaystyle \]}
\end{mapleinput}
\end{maplegroup}
\begin{maplegroup}
\begin{mapleinput}
\mapleinline{active}{1d}{assume(alpha, 'real'): assume(t, 'real'):}{\[\]}
\end{mapleinput}
\end{maplegroup}
\bigskip

Next, we define the map $\phi: \s \to \p$ and the Lie bracket $[ \cdot,\cdot ]$ (in Maple denoted by $LB(\cdot,\cdot)$):
\bigskip

\begin{maplegroup}
\begin{mapleinput}
\mapleinline{active}{1d}{Phi := X -> (1/2)*X+(1/2)*conjugate(Transpose(X)):}{\[\]}
\end{mapleinput}
\end{maplegroup}
%\noindent
%\begin{maplegroup}
%\begin{Maple Normal}{
%Lie bracket}\end{Maple Normal}
%\medskip
\begin{maplegroup}
\begin{mapleinput}
\mapleinline{inert}{1d}{LB := (X1, X2) -> X1.X2-X2.X1:}{\[\displaystyle \]}
\end{mapleinput}
\end{maplegroup}
\bigskip

Now, we define the inner product $\langle \cdot,\cdot \rangle_\g$ (in Maple denoted by $G(\cdot,\cdot)$) and
the unit vectors $H_0, H_1$ and $H = \cos(\alpha) H_0 + \sin(\alpha) H_1, T_H = \sin(\alpha) H_0 - \cos(\alpha) H_1$, $V,W,Z_0$ in the tangent space $\s_H = T_eS_H$ 
of the hypersurface $S_H$:
\bigskip

%\noindent
%\begin{maplegroup}
%\begin{Maple Normal}{
%Inner product $\langle \cdot,\cdot \rangle_\g$}\end{Maple Normal}
%\medskip

\begin{maplegroup}
\begin{mapleinput}
\mapleinline{inert}{1d}{G := (X1, X2) -> 2*Trace(X1.conjugate(Transpose(X2))):}{\[\displaystyle \]}
\end{mapleinput}
\end{maplegroup}
\begin{maplegroup}
\begin{mapleinput}
\mapleinline{inert}{1d}{H0 := Matrix([[1/2, 0, 0], [0, 0, 0], [0, 0, -1/2]]):}{\[\displaystyle \]}
\end{mapleinput}
\end{maplegroup}
\begin{maplegroup}
\begin{mapleinput}
\mapleinline{inert}{1d}{H1 := Matrix([[(1/6)*3^(1/2), 0, 0], \newline [0, -(1/3)*3^(1/2), 0], [0, 0, (1/6)*3^(1/2)]]):}{\[\displaystyle \]}
\end{mapleinput}
\end{maplegroup}
\begin{maplegroup}
\begin{mapleinput}
\mapleinline{inert}{1d}{H := cos(alpha)*H0+sin(alpha)*H1:}{\[\displaystyle \]}
\end{mapleinput}
\end{maplegroup}
\begin{maplegroup}
\begin{mapleinput}
\mapleinline{inert}{1d}{T_H := sin(alpha)*H0-cos(alpha)*H1:}{\[\displaystyle \]}
\end{mapleinput}
\end{maplegroup}
\begin{maplegroup}
\begin{mapleinput}
\mapleinline{active}{1d}{V := Matrix([[0, 1, 0], [0, 0, 0], [0, 0, 0]]):}{\[\displaystyle \]}
\end{mapleinput}
\end{maplegroup}
\begin{maplegroup}
\begin{mapleinput}
\mapleinline{inert}{1d}{ W := Matrix([[0, 0, 0], [0, 0, 1], [0, 0, 0]]):}{\[\displaystyle \]}
\end{mapleinput}
\end{maplegroup}
\begin{maplegroup}
\begin{mapleinput}
\mapleinline{inert}{1d}{Z0 := Matrix([[0, 0, 1], [0, 0, 0], [0, 0, 0]]):}{\[\displaystyle \]}
\end{mapleinput}
\end{maplegroup}
\medskip
\bigskip

The Riemannian curvature tensor $\langle R^S(X_1,X_2)X_2 , X_1 \rangle_\g$ in the ambient space $S$
(in Maple denoted by $R_S(X1,X2)$), the second fundamental form: $\langle \nabla_{X_1}^S T_H, X_2 \rangle_\g$
(in Maple denoted by $SecFF(X1,X2)$), the curvature tensor $\langle R^{S_H}(X_1,X_2)X_2 , X_1 \rangle_\g$
in the hypersurface $S_H$ (in Maple denoted by $R_{SH}(X1,X2)$ and the Ricci curvature $\Ric^{S_H}(X)$ (in Maple denoted by $Ric_{SH}(X) $ are introduced via the following lines:
\bigskip

%\noindent
%\begin{maplegroup}
%\begin{Maple Normal}{
%Curvature tensor in ambient space S: $\langle R^S(X1,X2)X2 , X1 \rangle_\g$}\end{Maple Normal}
%\medskip
\begin{maplegroup}
\begin{mapleinput}
\mapleinline{inert}{1d}{R_S := (X1, X2) -> -G(LB(LB(Phi(X1),Phi(X2)),\newline Phi(X2)), Phi(X1)):}{\[\displaystyle \]}
\end{mapleinput}
\end{maplegroup}
\begin{maplegroup}
\begin{mapleinput}
\mapleinline{inert}{1d}{SecFF := (X1, X2) -> (1/2)*G(Phi(X1),Phi(LB(X2,T_H))) \newline + (1/2)*G(Phi(X2),Phi(LB(X1,T_H))):}{\[\displaystyle \]}
\end{mapleinput}
\end{maplegroup}
\begin{maplegroup}
\begin{mapleinput}
\mapleinline{inert}{1d}{R_SH := (X1, X2) -> R_S(X1, X2) + SecFF(X1, X1)*\newline SecFF(X2, X2) - (SecFF(X1, X2))^2:}{\[\displaystyle \]}
\end{mapleinput}
\end{maplegroup}
\begin{maplegroup}
\begin{mapleinput}
\mapleinline{inert}{1d}{Ric_SH := X -> R_SH(V, X) + R_SH(I*V, X) + R_SH(W, X) \newline + R_SH(I*W, X) + R_SH(Z0, X) + R_SH(I*Z0, X) + R_SH(H, X):}{\[\displaystyle \]}
\end{mapleinput}
\end{maplegroup}
\bigskip

The relevant results are now obtained via the following lines:
\bigskip

\begin{maplegroup}
%\begin{Maple Normal}{
%$\Ric^{S_H}(aV+bW+cZ0+tH)$}\end{Maple Normal}
%\medskip
\begin{mapleinput}
\mapleinline{inert}{1d}{simplify(expand(Ric_SH(a*V+b*W+c*Z0+t*H)));}{\[\displaystyle \]}
%\mapleresult
\end{mapleinput}
\begin{maplelatex}
\begin{multline*}
-2*\cos(\alpha)*\sin(\alpha)*(|a|-|b|)*(|a|+|b|)*\sqrt{3}+ \\
(-2*|a|^2-2*|b|^2-4*|c|^2)*\cos(\alpha)^2-3*t^2-|a|^2-|b|^2+|c|^2
\end{multline*}
\end{maplelatex}
\end{maplegroup}
%\noindent
\begin{maplegroup}
%\begin{Maple Normal}{
%Proof that $S_H$ admits planes of positive curvature for all alpha in $(0,\pi/2]$}\end{Maple Normal}
%\medskip
\begin{mapleinput}
\mapleinline{inert}{1d}{simplify(expand(R_SH((2/3)^(1/2)*W+(1/3)*3^(1/2)*Z0, \newline -(2/3)^(1/2)*I*W+(1/3)*3^(1/2)*I*Z0)));}{\[\displaystyle \]}
\end{mapleinput}
%\mapleresult
\begin{maplelatex}
$$(4/9)*\sin(\alpha)*\cos(\alpha)*\sqrt{3}+(1/9)*\sin(\alpha)^2$$
\end{maplelatex}
\end{maplegroup}

\end{document}